\documentclass[11pt]{amsart}
\usepackage{amsmath}
\usepackage{amsmath,amssymb,amsthm}
\usepackage{amsfonts}
\usepackage{amssymb}
\usepackage{amsmath}
\usepackage{mathrsfs}
\usepackage{amsmath,amssymb}
\usepackage{amsmath}
\usepackage{graphicx}
\usepackage{amsmath,amssymb,amsthm}
\newcommand{\co}{\mathbb{C}}

\newcommand{\D}{{D}}
\newcommand{\Dn}{{D}^n}
\newcommand{\C}{{C}}
\newcommand{\Tb}{\overline{T}}
\newcommand{\Sb}{\overline{S}}
\newcommand{\zetab}{\bar{\zeta}}
\newcommand{\etab}{\bar{\eta}}
\newcommand{\bb}{\bar{b}}
\newcommand{\ab}{\bar{a}}
\newcommand{\zb}{\bar z}
\newcommand{\wb}{\bar w}
\newcommand{\1}{\mathbf{1}}
\newcommand{\Int}{\mbox{Int}}

\newcommand{\p}{\partial}
\newcommand{\pb}{\bar{\partial}}

\newtheorem{thm}{Theorem}[section]
\newtheorem{prop}[thm]{Proposition}
\newtheorem{lem}[thm]{Lemma}
\newtheorem{rem}[thm]{Remark}
\newtheorem{cor}[thm]{Corollary}

\begin{document}

\title[High-order Green Operator]{
High-order Green Operators on the Disk and the Polydisc }
\author[ Y. Liu, Z.H. Chen and Y.F. Pan]{ Yang Liu, Zhihua Chen and Yifei Pan}

\address{Department of Mathematics, Zhejiang Normal University, Jinhua 321004, China}
\email{liuyang4740@gmail.com}

\address{Department of Mathematics, Tongji University, Shanghai 200092, China}
\email{zzzhhc@tongji.edu.cn}

\address{Department of Mathematical Sciences, Indiana University-Purdue University Fort Wayne, Fort Wayne, Indiana 46805, USA.}
\email{pan@ipfw.edu}
\thanks{This work was supported by the National Natural Science Foundations of China (No. 11171255, 11101373), Doctoral Program Foundation of the Ministry of
Education of China (No. 20090072110053), and Zhejiang Innovation Project (No. T200905).}



\dedicatory{}

\begin{abstract}
 In this paper, we give the explicit expressions of high-order Green operators on the disk and the polydisc, and hence the kernel functions of high-order Green operators are also presented. As applications, we present the explicit integral expressions of all the solutions for linear high-order partial differential equations in the disk.
\end{abstract}

\maketitle
\section{ Preliminaries}

Singular integral is an important tool in harmonic analysis, complex analysis and Clifford analysis, et. al. There are many types of singular integrals investigated by numerous mathematicians. In \cite{qi1}, Calder\'on-Zygmund type singular integrals have been studied in theory and applications. For the singular integral operator, in \cite{qi4}, a boundedness criterion for the Cauchy singular integral operator in weighted grand Lebesgue spaces have been given. Moreover,
semi-Fredholm properties of certain singular integral operators, asymptotic invertibility of Toeplitz operators, weighted uniform convergence of the quadrature method for Cauchy singular integral equations, and Toeplitz and singular integral operators on general Carleson Jordan curves have been discussed in \cite{qi2}. The book
\cite{qi3} has been devoted to the Fredholm theory of singular integral operators with shifts on $L^p(\Gamma), ~(1<p<\infty)$, where $\Gamma$ is a Lyapunov curve in the complex plane which is homeomorphic to a circle or a segment.

In one complex dimension, the inverse operator (or Green operator) of the Cauchy-Riemann operator $\pb$ plays an important role in finding a solution of $\pb$ equation, see \cite{nw,nn,p}. Especially, high-order Green operators are used in \cite{p} to prove the general existence theorem for nonlinear partial differential systems of any order in one
complex variable. In this paper, we will give the explicit expression of high-order Green operators on the disk and the polydisc, and the kernel functions of high-order Green operators will be presented as well.

Let $\D$ be the closed disk $\{z\in\co||z|\leq R\}$ and $\C$ be its boundary $\{z\in\co||z|=R\}$. Unless otherwise state, in this paper, functions we consider will be complex valued and integrable with $D$.
$C^\alpha(\D)$ denotes the set of all functions $f$ on $\D$, where
$$H_\alpha[f]=\sup\{\frac{|f(z)-f(z')|}{|z-z'|^\alpha}\big|z,z'\in \D\}$$
is finite. For $f\in C^\alpha(\D)$ we define
$$||f||=|f|+(2R)^\alpha H_\alpha[f],$$ where $|f|$ denotes $\sup_{z\in D}|f(z)|$.
$C^k(\D)$ is the set of all functions $f$ on $\D$ whose $k$th order partial derivatives exist and are continuous,
$k\geq 0$ is an integer. $C^{k+\alpha}(\D)$ is the set of all functions $f$ on $\D$ whose $k$th order partial derivatives exist
 and belong to $C^\alpha(\D)$. For $f\in C^{k+\alpha}(\D)$, we have the definition in terms of $||\cdot||$:
 $$||f||^{(k)}=\max\limits_{i+j=k}\{||\partial^i\pb^jf||\}.$$
It should be pointed out that the function $||\cdot||^{(k)}$ on $C^{k+\alpha}(\D)$ is a semi-norm rather than a norm since $||f||^{(k)}=0$ if and only if $f$ is polynomial of degree of $k-1$.
The following operators are defined on $\D$ as in \cite{nw}:
\begin{equation} \label{eq:1}
\begin{split}
Tf(z)&=\frac{-1}{2\pi i}\int_\D\frac{f(\zeta)d\zetab \wedge d\zeta}{\zeta-z},~
\Tb f(z)=\frac{-1}{2\pi i}\int_\D\frac{f(\zeta)d\zetab \wedge d\zeta}{\zetab-\zb},\\
^2Tf(z)&=\frac{-1}{2\pi i}\int_\D\frac{f(\zeta)-f(z)}{(\zeta-z)^2}d\zetab \wedge d\zeta,~
^2\Tb f(z)=\frac{-1}{2\pi i}\int_\D\frac{f(\zeta)-f(z)}{(\zetab-\zb)^2}d\zetab \wedge d\zeta,\\
Sf(z)&=\frac{1}{2\pi i}\int_\C\frac{f(\zeta)d\zeta}{\zeta-z},
~\Sb f(z)=\frac{-1}{2\pi i}\int_\C\frac{f(\zeta)d\zetab}{\zetab-\zb},\\
S_bf(z)&=\frac{1}{2\pi i}\int_\C\frac{f(\zeta)d\zetab}{\zeta-z},~
\Sb_bf(z)=\frac{-1}{2\pi i}\int_\C\frac{f(\zeta)d\zeta}{\zetab-\zb},
 \end{split}
 \end{equation}
 where $\Tb(f)=\overline{T(\bar{f})}$, $\Sb(f)=\overline{S(\bar{f})}$, and $\Sb_b(f)=\overline{S_b(\bar{f})}$.
 More generally, if $\triangle$ is a closed bounded domain, then $T_\triangle f$ and $S_\triangle f$ are defined for continuous $f$ on $\triangle$ by
 $$T_\triangle f(z)=\frac{-1}{2\pi i}\int_\triangle\frac{f(\zeta)d\zetab \wedge d\zeta}{\zeta-z},~S_{\partial\triangle} f(z)=\frac{1}{2\pi i}\int_{\partial\triangle}\frac{f(\zeta)d\zeta}{\zeta-z}.$$

The fundamental properties between operators $T,~S$ are given by \cite{nw} as follows.
\begin{lem}\label{lem0}\cite{nw}
If $f\in C^1(\D)$, then
\begin{equation} \label{eq:2}
\begin{split}
T\pb f=f-Sf ~\mbox{on}~\Int(\D),\\
\Tb\p f=f-\Sb f ~\mbox{on}~\Int(\D).\nonumber
 \end{split}
 \end{equation}

\end{lem}
The smoothness properties of integral operator $T$ has been shown in \cite{nw}.
\begin{lem}\label{lem1}\cite{nw}
If $f\in C^\alpha(\D)$, then $Tf\in C^{1+\alpha}(\D)$. Moreover,
\begin{equation} \label{eq:2}
\begin{split}
\bar{\partial}Tf=f,~~~~~~~~{\partial}Tf=^2Tf,\nonumber
 \end{split}
 \end{equation}
and $$H_\alpha[^2Tf]\leq C_0H_\alpha[f]$$
where $C_0=\frac{12}{\alpha(1-\alpha)}$. If $f\in C^{k+\alpha}(\D)(k\geq 0)$, then $Tf\in C^{k+1+\alpha}(\D)$.
\end{lem}
Meanwhile, if replace $\partial, \pb, T, ^2T$ with $ \pb, \partial,\Tb, ^2\Tb$, respectively, then one can get the similar result as above lemma for $\Tb$.

\begin{lem}\label{lem2}\cite{p}
It holds for $l\geq 0$ that
\begin{equation} \label{eq:3}
\begin{split}
\int_\triangle\frac{(\zetab-\zb_0)^l}{\zeta-w}d\zetab \wedge d\zeta=\frac{-2\pi i}{l+1}(\wb-\zb_0)^{l+1},\nonumber
 \end{split}
 \end{equation}
where $\triangle=\{\zeta\in\co||\zeta-\zb_0|\leq r\}$.
\end{lem}
\begin{proof}
It is easy to get Lemma \ref{lem0} if we replace $\D,~\C$ with $\triangle,~\partial\triangle$, respectively. Apply Lemma \ref{lem0} to the function $(\zetab-\zb_0)^{l+1}$, we can obtain that
 \begin{equation}
\begin{split}
\int_\triangle\frac{(\zetab-\zb_0)^l}{\zeta-w}d\zetab \wedge d\zeta=&\frac{-2\pi i}{l+1}T_\triangle (\pb(\zetab-\zb_0)^{l+1})(w)\\
=&\frac{-2\pi i}{l+1}[(\zetab-\zb_0)^{l+1}-S_\triangle((\zetab-\zb_0)^{l+1})](w)\\
=&\frac{-2\pi i}{l+1}[(\wb-\zb_0)^{l+1}-\frac{1}{2\pi i}\int_{|\zeta-z_0|=r}\frac{(\zetab-\zb_0)^{l+1}}{\zeta-w}d\zeta]\\
=&\frac{-2\pi i}{l+1}[(\wb-\zb_0)^{l+1}-\frac{r^{2(l+1)}}{2\pi i}\int_{|\zeta-z_0|=r}\frac{1}{(\zeta-z_0)^{l+1}(\zeta-w)}d\zeta]\\
=&\frac{-2\pi i}{l+1}(\wb-\zb_0)^{l+1},\nonumber
 \end{split}
 \end{equation}
where the last equality comes from the residue theorem. The lemma is proved.
\end{proof}
\section{High-order Green operator on $\D$}
Denote $T^2=TT,~\Tb^2=\Tb\Tb$, similar notations for $T^\mu,\Tb^\nu,T^\mu\Tb^\nu$ for any $\mu,\nu>0$. In this section, we will get the explicit expression of $T^\mu\Tb^\nu f$.
\begin{thm}\label{thm0}
Given $f\in C^\alpha(\D)$, $T^kf,~\Tb^kf\in C^{k+\alpha}(\D)$ with integer $k>0$, and
\begin{equation} \label{eq:3}
\begin{split}
T^kf(z)=&\frac{(-1)^k}{(k-1)!\cdot 2\pi i}\int_\D\frac{(\zetab-\zb)^{k-1}f(\zeta)}{\zeta-z}d\zetab \wedge d\zeta, \end{split}
 \end{equation}
 \begin{equation} \label{eq:3.5}
\begin{split}
\Tb^kf(z)=&\frac{(-1)^k}{(k-1)!\cdot 2\pi i}\int_\D\frac{(\zeta-z)^{k-1}f(\zeta)}{\zetab-\zb}d\zetab \wedge d\zeta.
 \end{split}
 \end{equation}
\end{thm}

\begin{proof}
When $k=1$, (\ref{eq:3}) is obvious.

Now we assume (\ref{eq:3}) is valid when for $k-1$, i.e.,
\begin{equation} \label{eq:4}
\begin{split}
T^{k-1}f(z)=\frac{(-1)^{k-1}}{(k-2)!\cdot 2\pi i}\int_\D\frac{(\zetab-\zb)^{k-2}f(\zeta)}{\zeta-z}d\zetab \wedge d\zeta.\nonumber
 \end{split}
 \end{equation}
Hence, we can obtain with Lemma \ref{lem2} that
\begin{equation} \label{eq:5}
\begin{split}
T^kf(z)=&TT^{k-1}f(z)\\
       =&\frac{-1}{2\pi i}\int_\D\frac{T^{k-1}f(\zeta)}{\zeta-z}d\zetab \wedge d\zeta\\
       =&\frac{-1}{2\pi i}\int_\D\frac{\frac{(-1)^{k-1}}{(k-2)!\cdot 2\pi i}\int_\D\frac{(\etab-\zetab)^{k-2}f(\eta)}{\eta-\zeta}d\etab \wedge d\eta}{\zeta-z}d\zetab \wedge d\zeta\\
       =&\frac{(-1)^{k}}{(k-2)!\cdot (2\pi i)^2}\int_\D\int_\D\frac{(\etab-\zetab)^{k-2}d\zetab \wedge d\zeta}{(\eta-\zeta)(\zeta-z)}f(\eta) d\etab \wedge d\eta\\
=&\frac{(-1)^{k}}{(k-2)!\cdot (2\pi i)^2}\int_\D \frac{f(\eta) d\etab \wedge d\eta}{\eta-z}\int_\D(\etab-\zetab)^{k-2}(\frac{1}{\eta-\zeta}+\frac{1}{\zeta-z})d\zetab \wedge d\zeta \\
=&\frac{(-1)^{k}}{(k-2)!\cdot (2\pi i)^2}\int_\D\frac{f(\eta) d\etab \wedge d\eta}{\eta-z}[\int_\D\frac{(\etab-\zetab)^{k-2}}{\eta-\zeta}d\zetab \wedge d\zeta\\&+\int_\D\frac{(\etab-\zetab)^{k-2}}{\zeta-z}d\zetab \wedge d\zeta]\\
=&\frac{(-1)^{k}}{(k-2)!\cdot (2\pi i)^2}\int_\D\frac{f(\eta) d\etab \wedge d\eta}{\eta-z}[\int_\D\frac{\sum\limits_{l=0}^{k-2}\binom{k-2}{l}\etab^l(-\zetab)^{k-2-l}}{\eta-\zeta}d\zetab \wedge d\zeta\\&+\int_\D\frac{\sum\limits_{l=0}^{k-2}\binom{k-2}{l}\etab^l(-\zetab)^{k-2-l}}{\zeta-z}d\zetab \wedge d\zeta]\\
=&\frac{(-1)^{k}}{(k-2)!\cdot (2\pi i)^2}\int_\D\frac{\sum\limits_{l=0}^{k-2}\binom{k-2}{l}\etab^l f(\eta) d\etab \wedge d\eta}{\eta-z}[\int_\D\frac{(-\zetab)^{k-2-l}}{\eta-\zeta}d\zetab \wedge d\zeta\\&+\int_\D\frac{(-\zetab)^{k-2-l}}{\zeta-z}d\zetab \wedge d\zeta]\\
\nonumber
 \end{split}
 \end{equation}
 \begin{equation} \label{eq:5}
\begin{split}
=&\frac{(-1)^{k}}{(k-2)!\cdot (2\pi i)^2}\int_\D\frac{f(\eta)}{\eta-z}\sum\limits_{l=0}^{k-2}\binom{k-2}{l}\etab^l[\frac{2\pi i}{k-1-l}\bar{\eta}^{k-1-l}(-1)^{k-2-l}\\&+\frac{-2\pi i}{k-1-l}\bar{z}^{k-1-l}(-1)^{k-2-l}]d\etab \wedge d\eta\\
=&\frac{(-1)^{k}}{(k-2)!\cdot (2\pi i)^2}\int_\D\frac{-2\pi if(\eta)}{(\eta-z)(k-1)}\sum\limits_{l=0}^{k-1}\binom{k-1}{l}\etab^l[(-\etab)^{k-1-l}-(-\bar{z})^{k-1-l}]d\etab \wedge d\eta\\
=&\frac{(-1)^{k}}{(k-2)!\cdot (2\pi i)^2}\int_\D\frac{-2\pi if(\eta)}{(\eta-z)(k-1)}[(\etab-\etab)^{k-1}-(\etab-\zb)^{k-1}]d\etab \wedge d\eta\\
=&\frac{(-1)^{k}}{(k-2)!\cdot (2\pi i)^2}\int_\D\frac{2\pi if(\eta)}{(\eta-z)(k-1)}(\etab-\zb)^{k-1}d\etab \wedge d\eta\\
       =&\frac{(-1)^{k}}{(k-1)!\cdot 2\pi i}\int_\D\frac{(\etab-\zb)^{k-1}f(\eta)}{\eta-z} d\etab \wedge d\eta.\nonumber
 \end{split}
 \end{equation}
Thus, (\ref{eq:3}) is proved. For (\ref{eq:3.5}), the proof is similar, we omit here.
\end{proof}

\begin{rem}
From Theorem \ref{thm0}, we have $\Tb^k \bar{f}(z)=\overline{T^k{f}(z)}$.
\end{rem}
%%%%%%%%%%%%%%%%%%%%%%%%%%%%%%%%%%%%%%%%%%%%%%%%%%%%%%%%%%%%%%%%%%%%%%%%%%%%%%%%%%%%%%%%%%%%%%%

\begin{lem}\label{lem4}
For any $a,~b\in \D$ and $a\neq b$, integer $k>0$,
\begin{equation} \label{eq:8}
\begin{split}
\int_\D\frac{(\zeta-b)^{k-1}d\zetab \wedge d\zeta}{(\zeta-a)(\zetab-\bb)}=2\pi i\Big( C_1(a,b,k)+(a-b)^{k-1}\ln \frac{R^2-a\bb}{|a-b|^2}\Big),\nonumber
 \end{split}
 \end{equation}
 where $C_1(a,b,k)=\sum\limits_{l=1}^{k-1}(\frac{-b^l}{l})\Big(\sum\limits_{j=0}^{k-1-l}\binom{k-1}{j}a^{k-1-j}(-b)^j\Big)$ for $k>1$ and $C_1(a,b,1)=0$.
\end{lem}

\begin{proof}

For any small $\varepsilon>0$, denote $\D_a=\{|\zeta-a|<\varepsilon\},~\D_b=\{|\zeta-b|<\varepsilon\}$, $\C_a=\{|\zeta-a|=\varepsilon\},~\C_b=\{|\zeta-b|=\varepsilon\}$. Then for $k>1$,
 \begin{equation} \label{eq:9}
\begin{split}
\int_{\D\backslash \{\D_a\cup\D_b\}}&\frac{(\zeta-b)^{k-1}d\zetab \wedge d\zeta}{(\zeta-a)(\zetab-\bb)}\\
=&\int_{\D}d\bigg(\frac{(\zeta-b)^{k-1}\ln |\zeta-b|^2 d\zeta}{\zeta-a}\bigg)\\
&-\int_{\D_a}d\bigg(\frac{(\zeta-b)^{k-1}\ln |\zeta-b|^2 d\zeta}{\zeta-a}\bigg)\\
&-\int_{\D_b}d\bigg(\frac{(\zeta-b)^{k-1}\ln |\zeta-b|^2 d\zeta}{\zeta-a}\bigg)\\
=&\int_{\C}\frac{(\zeta-b)^{k-1}\ln |\zeta-b|^2 d\zeta}{\zeta-a}\\
&-\int_{\C_a}\frac{(\zeta-b)^{k-1}\ln |\zeta-b|^2 d\zeta}{\zeta-a}\\
&-\int_{\C_b}\frac{(\zeta-b)^{k-1}\ln |\zeta-b|^2 d\zeta}{\zeta-a},
 \end{split}
 \end{equation}
where the last equation comes from Stokes formula.
 \begin{equation} \label{eq:10}
\begin{split}
\int_{\D_a}&\frac{(\zeta-b)^{k-1}d\zetab \wedge d\zeta}{(\zeta-a)(\zetab-\bb)}\\
=&2i\int_{0}^{2\pi}\int_{0}^{\varepsilon}\frac{(re^{i\theta}+a-b)^{k-1}rdrd\theta}{re^{i\theta}(re^{-i\theta}+\ab-\bb)}\\
=&2i\int_{0}^{2\pi}\int_{0}^{\varepsilon}\frac{(re^{i\theta}+a-b)^{k-1}drd\theta}{r+(\ab-\bb)e^{i\theta}}\\
=&2i\int_{0}^{2\pi}\int_{0}^{\varepsilon}\frac{e^{i(k-1)\theta}\big(r+(\ab-\bb)e^{i\theta}-(\ab-\bb)e^{i\theta}+a-b\big)^{k-1}}{r+(\ab-\bb)e^{i\theta}}drd\theta\\
=&2i\int_{0}^{2\pi}\int_{0}^{\varepsilon}{e^{i(k-1)\theta}\sum\limits_{l=0}^{k-1}\binom{k-1}{l}\big(r+(\ab-\bb)e^{i\theta}\big)^{l-1}\big(-(\ab-\bb)e^{i\theta}+a-b\big)^{k-1-l}}drd\theta\\
=&2i\int_{0}^{2\pi}\int_{0}^{\varepsilon}{e^{i(k-1)\theta}\big(r+(\ab-\bb)e^{i\theta}\big)^{-1}\big(-(\ab-\bb)e^{i\theta}+a-b\big)^{k-1}}drd\theta\\
&+2i\int_{0}^{2\pi}\int_{0}^{\varepsilon}{e^{i(k-1)\theta}\sum\limits_{l=1}^{k-1}\binom{k-1}{l}\big(r+(\ab-\bb)e^{i\theta}\big)^{l-1}\big(-(\ab-\bb)e^{i\theta}+a-b\big)^{k-1-l}}drd\theta\\
=&2i\int_{0}^{2\pi}{e^{i(k-1)\theta}\big(-(\ab-\bb)e^{i\theta}+a-b\big)^{k-1}\ln\big(r+(\ab-\bb)e^{i\theta}\big)^{-1}}\Big|_{0}^{\varepsilon}d\theta\\
&+2i\int_{0}^{2\pi}\int_{0}^{\varepsilon}\frac{e^{i(k-1)\theta}}{l}{\sum\limits_{l=1}^{k-1}\binom{k-1}{l}\big(-(\ab-\bb)e^{i\theta}+a-b\big)^{k-1-l}\big(r+(\ab-\bb)e^{i\theta}\big)^{l}}\Big|_{0}^{\varepsilon}d\theta,
 \end{split}
 \end{equation}
which converges to 0 when $\varepsilon\rightarrow 0$.

 \begin{equation} \label{eq:11}
\begin{split}
\int_{\D_b}&\frac{(\zeta-b)^{k-1}d\zetab \wedge d\zeta}{(\zeta-a)(\zetab-\bb)}\\
=&2i\int_{0}^{2\pi}\int_{0}^{\varepsilon}\frac{(re^{i\theta})^{k-1}rdrd\theta}{re^{-i\theta}(re^{i\theta}+b-a)}\\
=&2i\int_{0}^{2\pi}\int_{0}^{\varepsilon}\frac{(re^{i\theta}+(b-a)-(b-a))^{k-1}drd\theta}{r+(b-a)e^{-i\theta}}\\
=&2i\int_{0}^{2\pi}\int_{0}^{\varepsilon}\frac{e^{i(k-1)\theta}\big(r+(b-a)e^{-i\theta}-(b-a)e^{-i\theta}\big)^{k-1}}{r+(b-a)e^{-i\theta}}drd\theta\\
=&2i\int_{0}^{2\pi}\int_{0}^{\varepsilon}{e^{i(k-1)\theta}\sum\limits_{l=0}^{k-1}\binom{k-1}{l}\big(r+(b-a)e^{-i\theta}\big)^{l-1}\big(-(b-a)e^{-i\theta}\big)^{k-1-l}}drd\theta\\
=&2i\int_{0}^{2\pi}\int_{0}^{\varepsilon}{e^{i(k-1)\theta}\big(r+(b-a)e^{-i\theta}\big)^{-1}\big(-(b-a)e^{-i\theta}\big)^{k-1}}drd\theta\\
&+2i\int_{0}^{2\pi}\int_{0}^{\varepsilon}{e^{i(k-1)\theta}\sum\limits_{l=1}^{k-1}\binom{k-1}{l}\big(r+(b-a)e^{-i\theta}\big)^{l-1}\big(-(b-a)e^{-i\theta}\big)^{k-1-l}}drd\theta\\
=&2i\int_{0}^{2\pi}{e^{i(k-1)\theta}\big(-(b-a)e^{-i\theta}\big)^{k-1}\ln\big(r+(b-a)e^{-i\theta}\big)^{-1}}\Big|_{0}^{\varepsilon}d\theta\\
&+2i\int_{0}^{2\pi}\int_{0}^{\varepsilon}\frac{e^{i(k-1)\theta}}{l}{\sum\limits_{l=1}^{k-1}\binom{k-1}{l}\big(-(b-a)e^{-i\theta}\big)^{k-1-l}\big(r+(b-a)e^{-i\theta}\big)^{l}}\Big|_{0}^{\varepsilon}d\theta,
 \end{split}
 \end{equation}
which also converges to 0 when $\varepsilon\rightarrow 0$.

 \begin{equation} \label{eq:12}
\begin{split}
\int_{\C_a}&\frac{(\zeta-b)^{k-1}\ln |\zeta-b|^2 d\zeta}{\zeta-a}\\
          =&\int_{0}^{2\pi}\frac{(\varepsilon e^{i\theta}+a-b)^{k-1}\ln |\varepsilon e^{i\theta}+a-b|^2 \varepsilon e^{i\theta}id\theta}{\varepsilon e^{i\theta}}\\
          =&\int_{0}^{2\pi}{(\varepsilon e^{i\theta}+a-b)^{k-1}\ln |\varepsilon e^{i\theta}+a-b|^2 id\theta},
 \end{split}
 \end{equation}
which converges to $2\pi i(a-b)^{k-1}\ln |a-b|^2$ when $\varepsilon\rightarrow 0$.

 \begin{equation} \label{eq:14}
\begin{split}
\int_{\C_b}&\frac{(\zeta-b)^{k-1}\ln |\zeta-b|^2 d\zeta}{\zeta-a}\\
          =&\int_{0}^{2\pi}\frac{(\varepsilon e^{i\theta})^{k-1}\ln |\varepsilon |^2 \varepsilon e^{i\theta}id\theta}{\varepsilon e^{i\theta}+b-a}\\
          =&\int_{0}^{2\pi}\frac{(\varepsilon e^{i\theta})^{k}\ln |\varepsilon |^2 id\theta}{\varepsilon e^{i\theta}+b-a},
 \end{split}
 \end{equation}
which converges to 0 when $\varepsilon\rightarrow 0$.

Thus, it comes from (\ref{eq:9})-(\ref{eq:14}) that
 \begin{equation} \label{eq:15}
\begin{split}
\int_\D\frac{(\zeta-b)^{k-1}d\zetab \wedge d\zeta}{(\zeta-a)(\zetab-\bb)}=&\int_{\C}\frac{(\zeta-b)^{k-1}\ln |\zeta-b|^2 d\zeta}{\zeta-a}-2\pi i(a-b)^{k-1}\ln |a-b|^2\\
=&\int_{\C}\frac{(\zeta-b)^{k-1}\Big(\ln |\zeta|^2+\ln(1-\frac{b}{\zeta})+\ln(1-\frac{\bb}{\zetab})\Big)}{\zeta-a}d\zeta-2\pi i(a-b)^{k-1}\ln |a-b|^2\\
=&\int_{\C}\frac{(\zeta-b)^{k-1}\ln R^2}{\zeta-a} d\zeta+\int_{\C}\frac{(\zeta-b)^{k-1}\ln(1-\frac{b}{\zeta})}{\zeta-a}d\zeta\\
&+\int_{\C}\frac{(\zeta-b)^{k-1}\ln(1-\frac{\bb}{\zetab})}{\zeta-a}d\zeta-2\pi i(a-b)^{k-1}\ln |a-b|^2\\
=&2\pi i(a-b)^{k-1}\ln R^2+I_1+I_2-2\pi i(a-b)^{k-1}\ln |a-b|^2.
\end{split}
 \end{equation}
$I_1$ is given as follows,
  \begin{equation} \label{eq:17}
\begin{split}
I_1=&\int_{\C}\frac{(\zeta-b)^{k-1}\ln(1-\frac{b}{\zeta})}{\zeta-a}d\zeta\\
   =&\int_{\C}\frac{(\zeta-b)^{k-1}}{\zeta-a}\sum\limits_{l=1}^{\infty}\frac{-1}{l}(\frac{b}{\zeta})^ld\zeta\\
   =&\sum\limits_{l=1}^{\infty}\frac{-b^l}{l}\int_{\C}\frac{(\zeta-b)^{k-1}}{\zeta-a}(\frac{1}{\zeta})^ld\zeta\\
   =&2\pi i\sum\limits_{l=1}^{k-1}(\frac{-b^l}{l})\Big(\sum\limits_{j=0}^{k-1-l}\binom{k-1}{j}a^{k-1-j}(-b)^j\Big).
\end{split}
 \end{equation}
 In fact, we have applied Residue theorem to get the last equation. Consider $\phi(\zeta)=\frac{(\zeta-b)^{k-1}}{\zeta-a}(\frac{1}{\zeta})^l$. The integral $\int_{\C}\frac{(\zeta-b)^{k-1}}{\zeta-a}(\frac{1}{\zeta})^ld\zeta$ equals $-2\pi i$ times the residue of $\phi$ at $\infty$ which is the opposite number
of the coefficient of $\frac{1}{\zeta}$ in the Rolland expansion of $\phi(\zeta)$.

   \begin{equation} \label{eq:18}
\begin{split}
\frac{(\zeta-b)^{k-1}}{\zeta-a}(\frac{1}{\zeta})^l=&\frac{(\zeta-b)^{k-1}}{\zeta(1-\frac{a}{\zeta})}(\frac{1}{\zeta})^l\\
=&\frac{1}{\zeta^{l+1}}\Big(\sum\limits_{p=0}^{\infty}(\frac{a}{\zeta})^p\Big)\Big(\sum\limits_{q=0}^{k-1}\binom{k-1}{q}\zeta^q(-b)^{k-1-q}\Big),\nonumber
\end{split}
 \end{equation}
 which implies that there is no $\frac{1}{\zeta}$ in the expansion unless $l\leq k-1$. Thus for any $1\leq l\leq k-1$, we get the coefficient of $\frac{1}{\zeta}$ in the Rolland expansion of $\phi(\zeta)$ as
   \begin{equation} \label{eq:19}
\begin{split}
\sum\limits_{j=0}^{k-1-l}\binom{k-1}{j}a^{k-1-j}(-b)^j.\nonumber
\end{split}
 \end{equation}
Hence,
$$\int_{\C}\frac{(\zeta-b)^{k-1}}{\zeta-a}(\frac{1}{\zeta})^ld\zeta=2\pi i\sum\limits_{j=0}^{k-1-l}\binom{k-1}{j}a^{k-1-j}(-b)^j,$$
and (\ref{eq:17}) can be obtained.

On the other hand, by Cauchy integral, we can obtain $I_2$ as follows,
  \begin{equation} \label{eq:21}
\begin{split}
I_2=&\int_{\C}\frac{(\zeta-b)^{k-1}\ln(1-\frac{\bb}{\zetab})}{\zetab-a}d\zeta\\
   =&\int_{\C}\frac{(\zeta-b)^{k-1}}{\zeta-a}\sum\limits_{l=1}^{\infty}\frac{-1}{l}(\frac{\bb}{\zetab})^ld\zeta\\
   =&\int_{\C}\frac{(\zeta-b)^{k-1}}{\zeta-a}\sum\limits_{l=1}^{\infty}\frac{-1}{l}(\frac{\bb}{R^2})^l\zeta^ld\zeta\\
   =&\sum\limits_{l=1}^{\infty}\frac{-1}{l}(\frac{\bb}{R^2})^l\int_{\C}\frac{(\zeta-b)^{k-1}}{\zeta-a}\zeta^ld\zeta\\
   =&2\pi i\sum\limits_{l=1}^{\infty}\frac{-1}{l}(\frac{\bb}{R^2})^l(a-b)^{k-1}a^l\\
   =&2\pi i(a-b)^{k-1}\ln(1-\frac{a\bb}{R^2}).
\end{split}
 \end{equation}
Therefore, from (\ref{eq:15}), (\ref{eq:17}), and (\ref{eq:21}), we prove that
 \begin{equation} \label{eq:22}
\begin{split}
\int_{\C}&\frac{(\zeta-b)^{k-1}\ln |\zeta-b|^2}{\zeta-a}d\zeta\\
=&2\pi i\sum\limits_{l=1}^{k-1}(\frac{-b^l}{l})\Big(\sum\limits_{j=0}^{k-1-l}\binom{k-1}{j}a^{k-1-j}(-b)^j\Big)+2\pi i(a-b)^{k-1}\ln(R^2-{a\bb}).\nonumber
\end{split}
 \end{equation}
 For $k=1$, one can easily yield from the above proof that $I_1=0$.
Combining (\ref{eq:15}), we get the lemma.
\end{proof}

\begin{lem}\label{lem5}
For any $a,~b\in \D$ and $a\neq b$, $l>0,~\nu>0$,
\begin{equation} \label{eq:23}
\begin{split}
\int_\C\frac{(\zetab-\bb)^l(\zeta-b)^{\nu-1}}{\zeta-a} d\zeta=2\pi iC_2(a,b,l,\nu),\nonumber
 \end{split}
 \end{equation}
 where $C_2(a,b,l,\nu)=\sum\limits_{0\leq p\leq l,0\leq q\leq \nu-1,p\leq q}\binom{l}{p}\binom{\nu-1}{q}R^{2p}(-\bb)^{l-p}(-b)^{\nu-1-q}a^{q-p}$.
\end{lem}

\begin{proof}
The computation is based on Residue theorem.
\begin{equation} \label{eq:24}
\begin{split}
\int_\C&\frac{(\zetab-\bb)^l(\zeta-b)^{\nu-1}}{\zeta-a} d\zeta\\
      =&\int_\C\frac{1}{\zeta-a}\Big( \sum\limits_{p=0}^{l}\binom{l}{p}\zetab^p(-\bb)^{l-p}\Big)\Big( \sum\limits_{q=0}^{\nu-1}\binom{\nu-1}{q}\zeta^q(-b)^{\nu-1-q}\Big) d\zeta\\
      =&\int_\C\frac{1}{\zeta-a}\Big( \sum\limits_{p=0}^{l}\binom{l}{p}\frac{R^{2p}}{\zeta^p}(-\bb)^{l-p}\Big)\Big( \sum\limits_{q=0}^{\nu-1}\binom{\nu-1}{q}\zeta^q(-b)^{\nu-1-q}\Big) d\zeta\\
      =&\int_\C\frac{1}{\zeta-a}\Big( \sum\limits_{p=0}^{l}\sum\limits_{q=0}^{\nu-1}\binom{l}{p}\binom{\nu-1}{q}R^{2p}(-\bb)^{l-p}(-b)^{\nu-1-q} \zeta^{q-p}\Big) d\zeta\\
      =& \sum\limits_{p=0}^{l}\sum\limits_{q=0}^{\nu-1}\binom{l}{p}\binom{\nu-1}{q}R^{2p}(-\bb)^{l-p}(-b)^{\nu-1-q}\int_\C\frac{1}{\zeta-a}\zeta^{q-p} d\zeta\\
      =&2\pi i \sum\limits_{0\leq p\leq l,0\leq q\leq \nu-1,p\leq q}\binom{l}{p}\binom{\nu-1}{q}R^{2p}(-\bb)^{l-p}(-b)^{\nu-1-q}a^{q-p},
 \end{split}
 \end{equation}
since if $q-p<0$, $\int_\C\frac{1}{\zeta-a}\zeta^{q-p} d\zeta=0$ and if $q-p\geq 0$, from the Residue theorem, $\int_\C\frac{1}{\zeta-a}\zeta^{q-p} d\zeta=2\pi i a^{q-p}$.

\end{proof}

\begin{lem}\label{lem6}
For any $a,~b\in \D$ and $a\neq b$, $\mu,~\nu>0$,
\begin{equation} \label{eq:26}
\begin{split}
\int_\D\frac{(\zetab-\ab)^{\mu-1}(\zeta-b)^{\nu-1}d\zetab \wedge d\zeta}{(\zeta-a)(\zetab-\bb)}=2\pi i C_3(a,b,\mu,\nu),\nonumber
 \end{split}
 \end{equation}
 where $C_3(a,b,\mu,\nu)=(\bb-\ab)^{\mu-1}\Big( C_1(a,b,\nu)+(a-b)^{\nu-1}\ln \frac{R^2-a\bb}{|a-b|^2}\Big)+\sum\limits_{l=1}^{\mu-1}
 \binom{\mu-1}{l}(\bb-\ab)^{\mu-1-l}\Big(\frac{-1}{l}(\ab-\bb)^{l}(a-b)^{\nu-1}-\frac{1}{l}C_2(a,b,l,\nu)\Big)$ for $\mu>1$ and $C_3(a,b,1,\nu)=C_1(a,b,\nu)+(a-b)^{\nu-1}\ln \frac{R^2-a\bb}{|a-b|^2}$.
\end{lem}

\begin{proof} For $\mu>1$, we have
\begin{equation} \label{eq:27}
\begin{split}
\int_\D&\frac{(\zetab-\ab)^{\mu-1}(\zeta-b)^{\nu-1}d\zetab \wedge d\zeta}{(\zeta-a)(\zetab-\bb)}\\
      =&\int_\D\frac{(\zetab-\bb+\bb-\ab)^{\mu-1}(\zeta-b)^{\nu-1}d\zetab \wedge d\zeta}{(\zeta-a)(\zetab-\bb)}\\
      =&\int_\D\frac{\sum\limits_{l=0}^{\mu-1}\binom{\mu-1}{l}(\zetab-\bb)^l(\bb-\ab)^{\mu-1-l}(\zeta-b)^{\nu-1}d\zetab \wedge d\zeta}{(\zeta-a)(\zetab-\bb)}\\
      =&(\bb-\ab)^{\mu-1}\int_\D\frac{(\zeta-b)^{\nu-1}d\zetab \wedge d\zeta}{(\zeta-a)(\zetab-\bb)}\\
      &+\sum\limits_{l=1}^{\mu-1}\binom{\mu-1}{l}(\bb-\ab)^{\mu-1-l}\int_\D\frac{(\zetab-\bb)^l(\zeta-b)^{\nu-1}d\zetab \wedge d\zeta}{(\zeta-a)(\zetab-\bb)}\\
      =&(\bb-\ab)^{\mu-1}\int_\D\frac{(\zeta-b)^{\nu-1}d\zetab \wedge d\zeta}{(\zeta-a)(\zetab-\bb)}\\
      &+\sum\limits_{l=1}^{\mu-1}\binom{\mu-1}{l}(\bb-\ab)^{\mu-1-l}\int_\D\frac{(\zetab-\bb)^{l-1}(\zeta-b)^{\nu-1}d\zetab \wedge d\zeta}{\zeta-a}\\
      =&I_3+I_4.\nonumber
 \end{split}
 \end{equation}
 From Lemma \ref{lem4}, $$I_3=2\pi i(\bb-\ab)^{\mu-1}\Big( C_1(a,b,\nu)+(a-b)^{\nu-1}\ln \frac{R^2-a\bb}{|a-b|^2}\Big).$$
For $f(\zeta)=\frac{1}{l}(\zetab-\bb)^{l}(\zeta-b)^{\nu-1}$, $\pb f(\zeta)=(\zetab-\bb)^{l-1}(\zeta-b)^{\nu-1}$. From Lemma \ref{lem1},
\begin{equation} \label{eq:28}
\begin{split}
I_4=&\sum\limits_{l=1}^{\mu-1}\binom{\mu-1}{l}(\bb-\ab)^{\mu-1-l}\int_\D\frac{(\zetab-\bb)^{l-1}(\zeta-b)^{\nu-1}d\zetab \wedge d\zeta}{\zeta-a}\\
      =&\sum\limits_{l=1}^{\mu-1}\binom{\mu-1}{l}(\bb-\ab)^{\mu-1-l}\int_\D\frac{\pb f(\zeta)}{\zeta-a}d\zetab \wedge d\zeta\\
      =&-2\pi i\sum\limits_{l=1}^{\mu-1}\binom{\mu-1}{l}(\bb-\ab)^{\mu-1-l}T(\pb f(\zeta))(a)\\
      =&-2\pi i\sum\limits_{l=1}^{\mu-1}\binom{\mu-1}{l}(\bb-\ab)^{\mu-1-l}\Big(f(a)-S(f(\zeta)(a))\Big)\\
      =&2\pi i\sum\limits_{l=1}^{\mu-1}\binom{\mu-1}{l}(\bb-\ab)^{\mu-1-l}\Big(\frac{-1}{l}(\ab-\bb)^{l}(a-b)^{\nu-1}-\frac{1}{l\cdot 2\pi i}\int_\C\frac{(\zetab-\bb)^l(\zeta-b)^{\nu-1}}{\zeta-a} d\zeta\Big)\\
      =&2\pi i\sum\limits_{l=1}^{\mu-1}\binom{\mu-1}{l}(\bb-\ab)^{\mu-1-l}\Big(\frac{-1}{l}(\ab-\bb)^{l}(a-b)^{\nu-1}-\frac{1}{l}C_2(a,b,l,\nu)\Big),\nonumber
 \end{split}
 \end{equation}
 where the last equation comes from Lemma \ref{lem5}. Furthermore, it is easy to see $I_4=0$ for $\mu=1.$
 Let $C_3(a,b,\mu,\nu)=(\bb-\ab)^{\mu-1}\Big( C_1(a,b,\nu)+(a-b)^{\nu-1}\ln \frac{R^2-a\bb}{|a-b|^2}\Big)+\sum\limits_{l=1}^{\mu-1}\binom{\mu-1}{l}(\bb-\ab)^{\mu-1-l}\Big(\frac{-1}{l}(\ab-\bb)^{l}(a-b)^{\nu-1}-\frac{1}{l}C_2(a,b,l,\nu)\Big)$. Then we prove the lemma.
\end{proof}

\begin{thm}\label{thm1}
Given $f\in C^\alpha(\D)$, $T^\mu\Tb^\nu f\in C^{\mu+\nu+\alpha}(\D)$ with $\mu,\nu>0$, and
\begin{equation}
\begin{split}
T^\mu\Tb^\nu f(z)=\frac{(-1)^{\mu}}{(\mu-1)!(\nu-1)!\cdot 2\pi i}\int_\D C_3(z,\eta,\mu,\nu) f(\eta)d\etab \wedge d\eta,\nonumber
 \end{split}
 \end{equation}
where the kernel function $C_3$ is given by Lemma \ref{lem6}.

\end{thm}

\begin{proof}

Given $f\in C^\alpha(\D)$, it is obvious that $T^\mu\Tb^\nu f\in C^{\mu+\nu+\alpha}(\D)$ by Lemma \ref{lem1}. From Theorem \ref{thm0}, we have

\begin{equation}
\begin{split}
T^\mu\Tb^\nu f(z)=&\frac{(-1)^\mu}{(\mu-1)!\cdot 2\pi i}\int_\D\frac{(\zetab-\zb)^{\mu-1}\Tb^\nu f(\zeta)}{\zeta-z}d\zetab \wedge d\zeta\\
                      =&\frac{(-1)^\mu}{(\mu-1)!\cdot 2\pi i}\int_\D\frac{(\zetab-\zb)^{\mu-1}\frac{(-1)^\nu}{(\nu-1)!\cdot 2\pi i}\int_\D\frac{(\eta-\zeta)^{\nu-1}f(\eta)}{\etab-\zetab}d\etab \wedge d\eta}{\zeta-z}d\zetab \wedge d\zeta\\
                      =&\frac{(-1)^{\mu+\nu}}{(\mu-1)!(\nu-1)!\cdot (2\pi i)^2}\int_\D\int_\D\frac{(\zetab-\zb)^{\mu-1}(\eta-\zeta)^{\nu-1} d\zetab \wedge d\zeta}{(\zeta-z)(\etab-\zetab)}f(\eta)d\etab \wedge d\eta\\
                      =&\frac{(-1)^{\mu+\nu}(-1)^\nu}{(\mu-1)!(\nu-1)!\cdot (2\pi i)^2}\int_\D\int_\D\frac{(\zetab-\zb)^{\mu-1}(\zeta-\eta)^{\nu-1} d\zetab \wedge d\zeta}{(\zeta-z)(\zetab-\etab)}f(\eta)d\etab \wedge d\eta\\
                      =&\frac{(-1)^{\mu}}{(\mu-1)!(\nu-1)!\cdot 2\pi i}\int_\D C_3(z,\eta,\mu,\nu) f(\eta)d\etab \wedge d\eta.\nonumber
\end{split}
 \end{equation}
\end{proof}

\begin{rem}
From above theorem, we have $\Tb^\mu T^\nu \bar{f}(z)=\overline{T^\mu\Tb^\nu {f}(z)}$.
\end{rem}

\begin{rem}
Theorem \ref{thm1} presents the expression of high-order Green operator and its kernel function on the disk. Here we give some special cases to shown the kernel in detail.

\begin{enumerate}
  \item$C_3(z,\eta,1,1)$=$\ln \frac{R^2-z\etab}{|z-\eta|^2}$;
  \item$C_3(z,\eta,1,2)$=$-z\eta+(z-\eta)\ln \frac{R^2-z\etab}{|z-\eta|^2}$;
  \item$C_3(z,\eta,2,1)$=$(\etab-\zb)\ln \frac{R^2-z\etab}{|z-\eta|^2}+2\etab-\zb$;
  \item$C_3(z,\eta,2,2)$=$(\etab-\zb)\big(-z\eta+(z-\eta)\ln \frac{R^2-z\etab}{|z-\eta|^2}\big)-|z-\eta|^2-|\eta|^2+\etab z-R^2$.
\end{enumerate}

\end{rem}

\cite{p} has presented the relationship between the corresponding norms for $f$ and $T^\mu\Tb^\nu f$ in the above theorem. We omit the proof here.
\begin{prop}\cite{p}
If $f\in C^\alpha(\D)$ and $\mu+\nu=m$, then
$$||T^\mu\Tb^\nu f||^{(m)}\leq 2^{\frac{(m-1)m}{2}}\big(C_4m+C_0+(m-1)C_5\big)^m||f||,$$
where $C_0=\frac{12}{\alpha(1-\alpha)}$, $C_4=\frac{2^{\alpha+1}}{\alpha}$ and $C_5=\frac{4}{\alpha(1-\alpha)}$.
\end{prop}
\section{High-order Green operator on $\Dn$}

Let $\Dn$ be the $n$-fold cartesian product of $D$, which is a closed ploydisc in $\mathbb{C}^n$ with radius $R$.
Suppose that $f$ is a complex-valued function defined on $\Dn$. We define $\Delta_if$ as a function on that subset $D_i$ of the $(n+1)-$fold product
$\D\times\cdots\times\D$ whose points $(z_1,...,z_{i-1},(z_i,z'_i),z_{i+1},...,z_n)$ satisfy $z_i\neq z'_i$. Then
$\Delta_if(z_1,...,z_{i-1},(z_i,z'_i),z_{i+1},...,z_n):=f(z_1,...,z_{i-1},z_i,z_{i+1},...,z_n)-f(z_1,...,z_{i-1},z'_i,z_{i+1},...,z_n)$.
For any distinct integers $i_1,...,i_k\in \{1,...,n\}$, we define $\Delta_{i_1\cdots i_k}f:=\Delta_{i_k}\Delta_{i_1\cdots i_{k-1}}f$. For simplicity, denote $Z_{i_1\cdots i_k}=(z_1,...,(z_{i_1},z'_{i_1}),\\...,(z_{i_k},z'_{i_k}),...,z_n)$ satisfy $z_{i_j}\neq z'_{i_j}$, $j=1,...,k.$ A $k$th-order $\alpha-$H\"older difference quotient is any expression $\delta_{i_1\cdots i_k}f$ defined by
$$\delta_{i_1\cdots i_k}f(Z_{i_1\cdots i_k})=\frac{\Delta_{i_1\cdots i_k}f(Z_{i_1\cdots i_k})}{|z_{i_1}-z'_{i_1}|^\alpha\cdots |z_{i_k}-z'_{i_k}|^\alpha}.$$
Set
$$H^{(k)}_\alpha[f]=\max\{|\delta_{i_1\cdots i_k}f|\big|i_1,...,i_k ~\mbox{distinct}\}.$$
We define $\mathcal{C}^\alpha(\Dn)$ as the set of those functions $f$ defined on $\Dn$ for which that $H^{(k)}_\alpha[f]$ are finite for all $k=0,...,n,$ where $H^{(0)}_\alpha[f]=|f|$.

Define $||\cdot||$ on the space $\mathcal{C}^\alpha(\Dn)$ by
$$||f||:=\sum_{k=0}^{n}\frac{{2R}^{k\alpha}}{k!}H^{(k)}_\alpha[f].$$
It is proved in 7.1b of \cite{nw} that $||\cdot||$ defined on $\mathcal{C}^\alpha(\Dn)$ is a norm.

Let $z=(z_1,...,z_n)\in \Dn$, the following operators are defined on $\Dn$:
\begin{equation} \label{}
\begin{split}
T_jf(z)&=\frac{-1}{2\pi i}\int_\D\frac{f(z_1,...,z_{j-1},\zeta,z_{j+1},...,z_n)d\zetab \wedge d\zeta}{\zeta-z_j},\\
\Tb_j f(z)&=\frac{-1}{2\pi i}\int_\D\frac{f(z_1,...,z_{j-1},\zeta,z_{j+1},...,z_n)d\zetab \wedge d\zeta}{\zetab-\zb_j},\nonumber
 \end{split}
 \end{equation}
 similar definitions are given for $S_j,~\Sb_j$. From Lemma \ref{lem0}, it is easy to see that for any $f\in C^1(\Dn)$,
$$T_j\pb_jf=f-S_jf.$$

Furthermore, it is given in \cite{nw} that
\begin{lem}\label{lem7}
If $f\in \mathcal{C}^\alpha(\Dn)$, then $T^if\in \mathcal{C}^\alpha(\Dn)$, $S^if\in\mathcal{C}^\alpha(\Dn)$ for all $i=1,...,n$. And there exist constants $C_6,~C_7$ such that
$$||T^if||\leq C_6R||f||,~~~~~~~||S^if||\leq C_7||f||.$$
\end{lem}
\begin{rem}
It should be noticed that the smoothness properties of the various integral operator $T^i,~\Tb^i$, $i=1,...,n$ defined for functions on $\D$ and $\Dn$ are different. $T^i,~\Tb^i$ are no longer smoothing order as in dimension one, see Lemma \ref{lem1}.
\end{rem}
Let vector index $\mu=(\mu_1,...,\mu_n),~\nu=(\nu_1,...,\nu_n),~\mathbf{1}=(1,...,1)\in \mathbb{N}^n$. $|\mu|=\sum\limits_{j=1}^{n}\mu_j$, $\mu!=\prod\limits_{j=1}^{n}\mu_j!$.
 For $z\in \Dn$, $f\in \mathcal{C}^\alpha(\Dn)$, denote the operator $T^\mu\Tb^\nu$ as follows,
$$T^\mu\Tb^\nu f(z)=T_1^{\mu_1}\cdots T_n^{\mu_n}\Tb_1^{\nu_1}\cdots\Tb_n^{\nu_n}f(z).$$
From Theorem \ref{thm1}, it is easy to get the explicit expression and the kernel function of $T^\mu\Tb^\nu f$ on $\Dn$.

\begin{cor}\label{thm2}
For $z\in \Dn$, given $f\in \mathcal{C}^\alpha(\Dn)$, then $T^\mu\Tb^\nu f\in \mathcal{C}^\alpha(\Dn)$ and
\begin{equation} \label{}
\begin{split}
T^\mu\Tb^\nu f(z)=C_8(\mu,\nu)\int_\D\cdots \int_\D \prod\limits_{j=1}^{n}C_3(z_j,\eta_j,\mu_j,\nu_j) f(\eta)d\etab_1 \wedge d\eta_1\cdots d\etab_n \wedge d\eta_n,\nonumber
 \end{split}
 \end{equation}
where $C_8(\mu,\nu)=\frac{(-1)^{|\mu|}}{(\mu-\1)!(\nu-\1)!\cdot (2\pi i)^n}$, and $C_3$ is given by Lemma \ref{lem6}.

\end{cor}

\begin{proof}
By the definition of $T^\mu\Tb^\nu f(z)$,
\begin{equation} \label{}
\begin{split}
T^\mu\Tb^\nu f(z)=&T_1^{\mu_1}\cdots T_n^{\mu_n}\Tb_1^{\nu_1}\cdots\Tb_n^{\nu_n}f(z)\\
                 =&T_1^{\mu_1}\Tb_1^{\nu_1}\cdots T_n^{\mu_n}\Tb_n^{\nu_n}f(z)\\
                 =&\frac{(-1)^{\mu_1}}{(\mu_1-1)!(\nu_1-1)!\cdot 2\pi i}\int_\D C_3(z_1,\eta_1,\mu_1,\nu_1) \\
                 &\times T_2^{\mu_2}\Tb_2^{\nu_2}\cdots T_n^{\mu_n}\Tb_n^{\nu_n}f(\eta_1,z_2,...,z_n)d\etab_1 \wedge d\eta_1\\
                 =&\frac{(-1)^{\mu_1+\mu_2}}{(\mu_1-1)!(\mu_2-1)!(\nu_1-1)!(\nu_2-1)!\cdot (2\pi i)^2}\\
                 &\times\int_\D\int_\D C_3(z_1,\eta_1,\mu_1,\nu_1)C_3(z_2,\eta_2,\mu_2,\nu_2) \\
                 &\times T_3^{\mu_3}\Tb_3^{\nu_3}\cdots T_n^{\mu_n}\Tb_n^{\nu_n}f(\eta_1,\eta_2,z_3,...,z_n)d\etab_1 \wedge d\eta_1\cdot d\etab_2 \wedge d\eta_2\\
                 =&\cdots\\
                 =&\frac{(-1)^{|\mu|}}{(\mu-\1)!(\nu-\1)!\cdot (2\pi i)^n}\\
                 &\times\int_\D\cdots \int_\D \prod\limits_{j=1}^{n}C_3(z_j,\eta_j,\mu_j,\nu_j) f(\eta)d\etab_1 \wedge d\eta_1\cdots d\etab_n \wedge d\eta_n.\nonumber
 \end{split}
 \end{equation}
\end{proof}

By Lemma \ref{lem7}, for any $f\in \mathcal{C}^\alpha(\Dn)$, one can easily get the following estimate for the norm of $f\in \mathcal{C}^\alpha(\Dn)$.
\begin{prop}
If $f\in \mathcal{C}^\alpha(\Dn)$ and $|\mu|+|\nu|=m$, then
$$||T^\mu\Tb^\nu f||\leq (C_6)^{m}||f||.$$
\end{prop}

\section{Applications}
As applications of the integral expressions for high-order Green operators, we can present all the solutions for some high-order Laplace equations, moreover, express all the solutions for linear high-order partial differential equation with integrals.

Given $D=\{z||z|\leq R\}=\{(x,y)|x^2+y^2\leq R^2\}$ as previous section.
Let $A(x,y): D\rightarrow \mathbb{R}$ be function of class $C^{\alpha}(D)$, $u(x,y):D\rightarrow \mathbb{R}$ be unknown function. Since $$\Delta=\frac{\partial^2}{\partial x^2}+\frac{\partial^2}{\partial y^2}=4\partial\pb$$ and from Lemma \ref{lem1} $$\pb T A=A; \p\Tb A=A,$$ then we have $$\Delta^2 T^2\Tb^2 A=16 A;~\Delta^2 \Tb^2 T^2 A=16 A.$$  $\Delta^2u=0$ is called biharmonic equation whose solutions can be described by $u=|z|^2h_1+h_2$, where $h_1,h_2$ are two harmonic functions satisfy Laplace's equation \cite{nature}. Then the solutions of 2-Laplace equation $u$:
\begin{equation} \label{}
\begin{split}\Delta^2u(x,y)=A(x,y) \nonumber\end{split}
 \end{equation}
given by
\begin{equation} \label{}
\begin{split}u=&\mathbf{Re}\Big(\frac{1}{32\pi i}\int_\D [(\etab-\zb)\big(-z\eta+(z-\eta)\ln \frac{R^2-z\etab}{|z-\eta|^2}\big)\\&-|z-\eta|^2-|\eta|^2+\etab z-R^2] A(\eta,\etab)d\etab \wedge d\eta\Big)\\&+|z|^2h_1+h_2. \nonumber
 \end{split}
 \end{equation}

The integral expressions can also be used to give all the solutions for linear partial differential equations on $\D$ of any order.  Let $H(D)$ be the set of all holomorphic functions on $\D$ and denote $T^0f=f$. We have the following results.
\begin{lem}\label{lem8}
Given $\mu,\nu> 0$, the solutions of
\begin{equation} \label{15}
\begin{split}
\p^\mu\pb^\nu u(z,\zb)=0 
\end{split}
 \end{equation}
can be given by
\begin{equation} \label{}
\begin{split}
u=\sum_{j=0}^{\nu-1}T^j(g_j)+T^\nu\big(\sum_{i=0}^{\mu-1}\Tb^i(\bar{f}_i)\big), \nonumber
\end{split}
 \end{equation}
 where $f_i,~g_j\in H(D),~i=0,...,\mu-1;~j=0,...,\nu-1$.
\end{lem}

\begin{proof}
It is well known that all the solutions of $\pb u=0$ can be given by any $u\in H(D)$ and the solutions of $\p u=0$ can be given by $u$ with any $\bar{u}\in H(D)$.
Consider the equation (\ref{15}), we have
$$\p^{\mu-1}\pb^\nu u=\bar{f}_0,~\forall f_0\in H(D),$$
which means that
$$\p^{\mu-2}\pb^\nu u=\bar{f}_1+\Tb(\bar{f}_0),~\forall f_0,f_1\in H(D).$$
Similarly, we have
$$\p^{\mu-3}\pb^\nu u=f_2+\Tb(\bar{f}_1)+\Tb^2(\bar{f}_0),~\forall f_0,f_1,f_2\in H(D).$$
By iteration, one has
$$\pb^\nu u=\sum_{i=0}^{\mu-1}\Tb^i(\bar{f}_{\mu-1-i}),~\forall f_i\in H(D),i=0,...,\mu-1.$$
Furthermore,
$$\pb^{\nu-1} u=g_0+T\sum_{i=0}^{\mu-1}\Tb^i(\bar{f}_{\mu-1-i}),~\forall g_0,f_i\in H(D),i=0,...,\mu-1,$$
and then
$$\pb^{\nu-2} u=g_1+T(g_0)+T^2\sum_{i=0}^{\mu-1}\Tb^i(\bar{f}_{\mu-1-i}),~\forall g_0,g_1,f_i\in H(D),i=0,...,\mu-1.$$
By iteration, we can conclude
$$u=\sum_{j=0}^{\nu-1}T^j(g_{\nu-1-j})+T^\nu\sum_{i=0}^{\mu-1}\Tb^i(\bar{f}_{\mu-1-i}),~\forall g_j,f_i\in H(D),i=0,...,\mu-1,~j=0,...,\nu-1.$$
For simplicity, we replace the index and prove the lemma. Using Theorems \ref{thm0} and \ref{thm1}, we can give the solutions of equation (\ref{15}) by integral.

\end{proof}
%%%%%%%%%%%%%%%%%%%%%%%%%%%%%%%%%%%%%%%%%%%%%%%%%%%%%%%%%%%%%%%%%%%%%%%%%%%%%%%%%%%%%%%%%%%%%%%%%

From Lemma \ref{lem8}, we can express all the solutions for linear high-order partial differential equation
$$\p^\mu\pb^\nu u(z,\zb)=A(z,\zb)$$
as
\begin{equation} \label{16}
\begin{split}u=\sum_{j=0}^{\nu-1}T^j(g_{j})+T^\nu\sum_{i=0}^{\mu-1}\Tb^i(\bar{f}_{i})+T^\nu\Tb^\mu(A),\end{split}
 \end{equation}
where $g_j,f_i\in H(D),i=0,...,\mu-1,~j=0,...,\nu-1.$ Using Theorems \ref{thm0} and \ref{thm1}, the integral expressions of solutions can be given.

For simplicity, we denote
\begin{equation} \label{18}
\begin{split}
&G(z,\zeta, l)=\frac{(-1)^l(\zetab-\zb)^{l-1}}{2\pi i(l-1)!(\zeta-z)};\\
&G(z,\zeta,\mu,\nu)=\frac{(-1)^{\mu}}{ 2\pi i(\mu-1)!(\nu-1)!} C_3(z,\zeta,\mu,\nu). 
\end{split}
 \end{equation}
 From (\ref{16}), we have
 \begin{equation} \label{19}
\begin{split}
u=g_0+\sum_{j=1}^{\nu-1}T^j(g_j)+T^\nu \bar{f}_0+T^\nu\big(\sum_{i=1}^{\mu-1}\Tb^i(\bar{f}_i)\big)+T^\nu\Tb^\mu(A). 
\end{split}
 \end{equation}
Combining (\ref{18}) and (\ref{19}), we have the following theorem.
\begin{thm}
Given $\mu,\nu> 0$, the solutions of
\begin{equation} \label{}
\begin{split}
\p^\mu\pb^\nu u(z,\zb)=A(z,\zb) \nonumber
\end{split}
 \end{equation}
with $A(z,\zb)\in C^{\alpha}(D)$ can be given by
\begin{equation} \label{}
\begin{split}
u(z,\zb)=&g_0(z)+\int_\D\Big(\sum_{j=1}^{\nu-1}G(z,\zeta, j)g_j(\zeta)+G(z,\zeta, \nu)\bar{f}_0(\zeta)\\&
+\sum_{i=1}^{\mu-1}G(z,\zeta,\nu,i)\bar{f}_i(\zeta)+G(z,\zeta,\nu,\mu)A(\zeta,\zetab)\Big)d\zetab\wedge d\zeta , \nonumber
\end{split}
 \end{equation}
 where $f_i,~g_j\in H(D),~i=0,...,\mu-1;~j=0,...,\nu-1$.

\end{thm}

\section{Conclusion}

We have established the explicit expressions for high-order Green operators on the disk in $\mathbb{C}$ and the polydisc in $\mathbb{C}^n$. As applications, we have presented all the solutions for biharmonic equations and high-order partial differential equations in the disk. The same method works identically in $\mathbb{R}^n$ through Clifford analysis and the results will be presented in the forthcoming paper in a near future.

\end{document}